\documentclass[11pt]{article}
\usepackage{amsfonts,amsmath,amssymb,amsthm, amscd}
\usepackage[dvips]{epsfig,graphics}

\numberwithin{equation}{section}

\newtheorem{theorem}{Theorem}[section]
\newtheorem{prop}[theorem]{Proposition}
\newtheorem{lemma}[theorem]{Lemma}
\newtheorem{cor}[theorem]{Corollary}
\newtheorem{conj}[theorem]{Conjecture}

\theoremstyle{definition}
\newtheorem{definition}[theorem]{Definition}
\newtheorem{example}[theorem]{Example}
\newtheorem{remark}[theorem]{Remark}

\def\<{{\langle}}
\def\>{{\rangle}}

\def\D{{\Delta}}
\def\e{{\epsilon}}

\def\l{{\lambda}}
\def\L{\Lambda}
\def\s{\sigma}
\def\Si{\Sigma}

\def\Rd{{\cal R}_d}

\def\ni{\noindent}
\def\bs{\bigskip}
\def\Z{\mathbb Z}

\def\S{\mathbb S}
\def\Q{\mathbb Q}

\def\doteq{\ {\buildrel \cdot \over =}\ }
\DeclareMathOperator{\rk}{rk}

\begin{document}

\title{Twisted Alexander Invariants of Twisted Links}

\author{Daniel S. Silver \and Susan G. Williams\thanks{Both authors partially supported by NSF grant
DMS-0706798.} \\ {\em
{\small Department of Mathematics and Statistics, University of South Alabama}}}

\maketitle 

\begin{abstract} Let $L= \ell_1 \cup \cdots \cup \ell_{d+1}$ be an oriented link in $\S^3$, and let $L(q)$ be the $d$-component link $\ell_1 \cup \cdots \cup \ell_d$ regarded in the homology 3-sphere that results from performing $1/q$-surgery on $\ell_{d+1}$. Results about the Alexander polynomial and twisted Alexander polynomials of $L(q)$ corresponding to finite-image representations are obtained. The behavior of the invariants as $q$ increases without bound is described.\medskip


\end{abstract} 


\section{Introduction.} Let $L = \ell_1 \cup \cdots \cup \ell_{d+1}$ be an oriented 
$(d+1)$-component link in the 3-sphere $\S^3$. For any positive integer $q$, let $L(q)$ be the $d$-component link $\ell_1 \cup \cdots \cup \ell_d$ regarded in the homology 3-sphere that results from {\it $1/q$-surgery} on $\ell_{d+1}$, removing a closed tubular neighborhood of $\ell_{d+1}$ and then replacing it in such a way that its meridian wraps once about the meridian of $\ell_{d+1}$ and $q$ times around the longitude. If $\ell_{d+1}$ is trivial and bounds a disk $D$, then $L(q)$ is a link in $\S^3$, and it can be obtained from $\ell_1 \cup \cdots \cup \ell_d$ by giving $q$ full twists to those strands that pass through $D$. 

In previous work \cite{swJLM}, we considered the multivariable Alexander polynomial $\D_{L(q)}$ and the limiting behavior of its Mahler measure as $q$ increases without bound. The case in which $\ell_{d+1}$ has zero linking number with each of the other components of $L$ was treated separately and combinatorially. The first goal here is to provide a more topological perspective of this case. The second goal is to generalize our twisting results for Alexander polynomials twisted by representations of the link group in a finite group.

We are grateful to Stefan Friedl and Jonathan Hillman for suggestions and helpful remarks. 

\section{Twisting and Mahler measure.} We survey some results about Mahler measure and Alexander polynomials, and offer motivation for the results that follow. 

\begin{definition} The {\it Mahler measure} $M(f)$ of a nonzero integral polynomial in $d \ge 1$ variables is $$M(f) = \exp \int_0^1 \cdots \int_0^1 log|f(e^{2 \pi \theta_1}, \ldots, e^{2 \pi \theta_d})| d \theta_1 \cdots d \theta_d.$$ The Mahler measure of the zero polynomial is $0$. \end{definition}

\begin{remark} (1) In the case of a polynomial in a single variable, Jensen's formula implies that $M(f)$ is equal to 
$$|c| \prod \max \{1, |r_i|\},$$ where $c$ is the leading coefficient of $f$ and 
the $r_i$ are the zeros (with possible multiplicity) of $f$. A proof of this and the following fundamental facts about Mahler measure can be found in \cite{ew}.\smallskip

\item{}(2) $M(fg) = M(f)M(g)$ for any nonzero polynomials $f, g$. We define the Mahler measure of a rational function $f/g$, $g\ne0$,  to be $M(f)/M(g)$. \smallskip

\item{}(3) $M(f) = 1$ if and only if $f$ is the product of a unit and generalized cyclotomic polynomials.  A {\it generalized cyclotomic polynomial} is a cyclotomic polynomial evaluated at a monomial; e.g., 
$ (t_1t_2)^2- t_1t_2 +1.$ \smallskip

\item{}(4) The Mahler measure of a nonzero Laurent polynomial can be defined either directly from the definition or by normalizing, multiplying by a monomial, so that all exponents are nonnegative.

\end{remark}

The Mahler measure of the Alexander polynomial of a link provides a measure of growth of homology torsion of its finite abelian branched covers. 
We identify $H_1(\S^3 \setminus L)$ with the free abelian multiplicative group $\Z^d$ generated by $t_1, \ldots, t_d$.

\begin{theorem} \label{growth} \cite{swTOP2} Let $L\subset \S^3$ be an oriented $d$-component link with nonzero Alexander polynomial $\D_L$. Then 
$$\log M(\D_L) = \limsup_{\< \L \> \to \infty} {1 \over |\Z^d/\L|} \log b_\L,$$
where $\L$ is a finite-index subgroup of $\Z^d$, $\< \L \> =\min \{|v| \mid v \in \L\setminus \{0\} \}$, $b_\L$ is the order of the torsion subgroup of 
$H_1(M_\L, \Z)$ and $M_\L$ is the associated abelian cover of $\S^3$ branched over $L$. In the case that $L$ is a knot (that is, $d=1$), limit superior can be replaced by an ordinary limit. \end{theorem}

We recall that the Alexander polynomial $\D_L$ is the first term in a sequence of polynomial invariants that are successive divisors (see below). The authors conjectured \cite{swTOP2} that when $\D_L$ vanishes, the conclusion of Theorem \ref{growth} holds with $\D_L$ replaced by the first nonvanishing higher Alexander polynomial. The conjecture was recently proved by T. Le \cite{le}. 

Much of the interest in Mahler measures of Alexander polynomials is motivated by an open question posed by D.H. Lehmer in 1933. \smallskip

\ni {\bf Lehmer's Question.}  Is 1 a limit point of the set of Mahler measures of integral polynomials in a single variable? \smallskip 

There are no known values in this set between 1 and the Mahler measure $ M({\cal L}) = 1.17628\ldots$ of a polynomial found by Lehmer,
$${\cal L}(t)= t^{10}+t^9-t^7-t^6-t^5-t^4-t^3+t+1.$$

Lehmer's polynomial ${\cal L}(t)$ has a single root outside the unit circle. When $t$ is replaced by $-t$, the result is the Alexander polynomial of a fibered, hyperbolic knot, the $(-2, 3, 7)$-pretzel knot. 

A method for constructing many polynomials with small Mahler measure greater than 1 is to begin with a product of cyclotomic polynomials and perturb the middle coefficient. Lehmer's polynomial arises as  $\phi_1^2 \phi_2^2 \phi_3^2\phi_6 - t^5$, where $\phi_i$ is the $i$th cyclotomic polynomial. 
In fact, all known Mahler measures less than 1.23 have been found in this way \cite{mpvMC}. 

Our principal motivation for studying the twisting effects on Alexander polynomials comes from a topological analogue of the above procedure: Starting with a periodic $n$-braid (the closure of which has an Alexander polynomial with Mahler measure equal to 1), one perturbs it by twisting $m$ consecutive strands, where $m < n$, and then forms the closure. Knots $k$ obtained this way are known as {\it twisted torus knots} \cite{dAGT}, and the Mahler measure of $\D_k$ is often small but greater than 1. In particular, the closure of the 3-braid $(\s_1\s_2 )^7 \s_1^{-2}$ is the $(-2, 3, 7)$-pretzel knot. 

Observe that $(\s_1\s_2 )^7 \s_1^{-2} = (\s_1\s_2 )^6 \s_1\s_2\s_1^{-2}$ is the product of the 3-braid $B= \s_1\s_2\s_1^{-2}$ and two full twists on all strands. Let $L$ be the 2-component link $\ell_1 \cup \ell_2$, where $\ell_1$ is the closure of $B$ while $\ell_2$ is an encircling unknot. 
With orientations chosen appropriately, the $(-2, 3, 7)$-pretzel knot is $L(2)$, the result of $1/2$-surgery on $\ell_2$. 

Since $L(2)$ is a perturbation of the closure of the periodic braid $(\s_1\s_2)^7$, we might expect that replacing $(\s_1\s_2)^6$ with higher powers $(\s_1\s_2)^{3q}$ would produce twisted torus knots with small Mahler measure. Indeed, calculations suggest that $M(\D_{L(q)})$ converges to a relatively low value of $1.285\ldots$. The following theorem shows that this is in fact so, and moreover the limit is the Mahler measure of $\D_L$. 

\begin{theorem} \label{limit} \cite{swJLM} Let $L= \ell_1 \cup \cdots \cup \ell_{d+1}$ be an oriented link in $\S^3$. If $\ell_{d+1}$ has non-zero linking number with some other component, then $$\lim_{q \to \infty} M(\D_{L(q)}) = M(\D_L).$$ \end{theorem} 

In the case that $\ell_{d+1}$ has zero linking number with all other components of $L$, the Mahler measures of $\D_{L(q)}$ can increase without bound. However, $(1/q)M(\D_{L(q)})$ converges. In fact, the polynomials $(1/q)\D_{L(q)}$ converge in a strong sense, as we argued combinatorially in \cite{swJLM}. In the next section, we examine this case more closely, from a topological perspective. 


\section{Twisting about an unlinked component.}

Alexander polynomials are defined for any finitely presented group $G$ and epimorphism $\e: G \to \Z^d$, $d \ge 1$. We regard $\Z^d$ as a multiplicative free abelian group with generators $t_1, \ldots, t_d$. One considers the kernel $K$ of $\e$. Its abelianization $K^{ab}$ is a finitely generated $\Rd$-module, where $\Rd = \Z[t_1^{\pm 1}, \ldots, t_d^{\pm 1}]$ is the ring of Laurent polynomials in $t_1, \ldots, t_d$. As $\Rd$ is a Noetherian unique factorization domain, a sequence of elementary ideals $E_k(K^{ab})$ is defined for $k \ge 0$.
The greatest common divisor of the elements of $E_k(K^{ab})$ is the $k$th {\it Alexander polynomial} of $(G, \e)$, denoted here by $\D^\e_{G, k}$. These polynomials, which are defined up to multiplication by units of $\Rd$,  form a sequence of successive divisors. The 0th polynomial is called the {\it Alexander polynomial} of $(G, \e)$ and we denote it more simply by $\D^\e_G$. 

For purposes of computation, one considers a presentation of $G$: 
\begin{equation} \label{presentation} \<x_0, x_1 \ldots, x_n \mid r_1, \ldots, r_m\>.\end{equation} With no loss of generality, we assume that $n \le m$ and $\e(x_0)=t_1$. 

The epimorphism pulls back to the free group $F$ generated
by $x_0, \ldots, x_n$, inducing a unique extension to a map of group rings $\Z[F] \to \Z[\Z^d]$ that, by abuse of notation, we denote also by $\e$.
One forms an {\it Alexander matrix} of the presentation: 
\begin{equation} \label{Alex.matrix1} A^\e =\bigg[\e \bigg({{\partial r_i}\over {\partial x_j}}\bigg)\bigg]_{1\le i, j \le n}, \end{equation} where ${\partial r_i}/{\partial x_j}$ are Fox partial derivatives (see, for example, \cite{kaw}). 
Then $\D_{G,k}^\e$ is equal to the greatest common divisor of the $(n-k)$-minors of $A^\e$, provided that $d=1$; when $d>1$, the result is divisible by $t_1-1$, and we must divide by it. 

Pairs $(G, \e)$ as above are objects of a category, morphisms being homomorphims $f: G \to G'$ such that $\e' \circ f = \e$. If $f: (G, \e) \to (G', \e')$ is a morphism, then $\D_{G'}^{\e'}$ divides $\D_G^\e$. In particular, $\D_{G'}^{\e'}$ is equal to $\D_G^\e$ up a unit factor of $\Rd$ (denoted $\D_{G'}^{\e'} \doteq \D_G^\e$) whenever $f$ is an isomorphism.

When $G$ is the group of an oriented link $L\subset \Sigma$ of $d$-components in a homology 3-sphere, there is a natural augmentation $\e: G \to \Z^d$ that maps the meridianal generators of the $i$th component to $t_i$, for $1 \le i \le d$. The 
Alexander polynomial $\D^\e_G$, an invariant of the link, is denoted here by $\D_L$. 

Assume that $L = \ell_1 \cup \cdots \cup \ell_{d+1}$ is an oriented 
$(d+1)$-component non-split link in the 3-sphere $\S^3$ such that $\ell_{d+1}$ has zero linking number with $\ell_1,\dots, \ell_d$. Let
\begin{equation} \label{Wirt} \< x_0, x_1, \ldots, x_n \mid r_1, \ldots, r_n\> \end{equation} 
be a Wirtinger presentation for $\pi_L = \pi_1(\S^3 \setminus L)$. {\sl We will assume throughout, without loss of generality,  that $x_n$ corresponds to a meridian of $\ell_{d+1}$.}  

Note that we have omitted a Wirtinger relation in the presentation (\ref{Wirt}). It is well known that any single Wirtinger relation is a consequence of the remaining ones. {\sl We assume throughout, again without loss of generality, that the omitted relation does not involve meridianal generators of $\ell_{d+1}$. }


The longitude of any component of $L$ is a simple closed curve in the torus boundary of a neighborhood of the component that is null-homologous in the link complement. Up to conjugacy, the longtitude represents an element of $\pi_L$, and a  representative word in the Wirtinger presentation associated to a diagram  can be read by tracing around the component of the link, recording generators or their inverses as we pass under arcs, and finally multiplying by an appropriate power of the generator corresponding to the arc where we began so that the element represented has trivial abelianization. 

\begin{lemma} \label{redundant}  Let $\l\in \pi_L$ be the homotopy class of the longitude of $\ell_{d+1}$. For any integer $q$, one of the relators 
$r_1, \ldots, r_n$ in the quotient group presentation 

\begin{equation} \label{surgery.presentation} \< x_0, x_1, \ldots, x_n \mid r_1, \ldots, r_n, \l^qx_n\>  \end{equation} is redundant.

\end{lemma}


\begin{proof} For the purpose of the proof, relabel the meridianal generators of $\ell_{d+1}$ by $w_1 (=x_n), \ldots, w_m$ such that $w_{j+1}$ follows $w_j$ as we travel along $\ell_{d+1}$ in its preferred direction  (subscripts $i$ regarded modulo $m$). Some arc of the link diagram passes between $w_j$ and $w_{j+1}$. The Wirtinger relations at such crossings allow us to express $w_2$ as a conjugate of $x_n$, and then $w_3$ as a conjugate of $x_n$, and so forth. The final relator, which expresses $x_n$ as a conjugate of itself, is $x_n \l x_n^{-1} \l^{-1}$. Substituting $x_n = \l^{-q}$  trivializes the relator, and hence it is redundant in the presence of the other relators.
\end{proof}

We denote by $L_0$ the oriented $d$-component sublink $\ell_1 \cup \cdots \cup \ell_d \subset L$. We denote by $X_{L_0}$ the exterior of $L_0$. Performing $1/q$-surgery on $\ell_{d+1}\subset \S^3$ yields a homology sphere $\Si(q)$. We regard $\ell_1 \cup \cdots \cup \ell_d$ as an oriented link $L(q) \subset \Si(q)$. When $\ell_{d+1}$ is unknotted, $\S(q) = \S^3$, and $L(q)$ is the link obtained from $L_0$ by giving $q$ full twists to those strands passing through a disk with boundary $\ell_{d+1}$. It is clear that (\ref{surgery.presentation}) is a presentation of $\pi_{L(q)}= \pi_1(\Si(q)\setminus L(q))$. 
An Alexander matrix $A^\e$ for $L(q)$ results from the Alexander matrix $A^\e$ for $L$ by adjoining a row consisting of the Fox partial derivatives of the added relation $\l^qx_n$. 
By adding nugatory crossings to $\ell_{d+1}$, we can arrange that the word representing the longitude $\lambda$ does not contain $x_n$ or $x_n^{-1}$.  Then in the new row we see $1$ in the column corresponding to $x_n$. In other columns we see an element of $\Rd$  multiplied by the integer $q$. The reason is the following. Assume that $x_{i_1}^{e_1}\cdots x_{i_m}^{e_m}$ represents $\l$, where each $e_j= 1$ or $-1$. 
For any $j = 1, \ldots, m$, the generator $x_{i_j}^{e_j}$ occurs $q$ times, and it contributes
$\e(x_{i_1}^{e_1}\ldots x_{i_{j-1}}^{e_{j-1}})(1+ \e(\l)+\ldots + \e(\l^{j-1}))= q \cdot \e(x_{i_1}^{e_1}\ldots x_{i_{j-1}}^{e^{j-1}})$, if $e_j = 1$; it is equal to $-\e(x_{i_1}^{e_1}\ldots x_{i_j}^{e_j})(1+ \e(\l)+\ldots + \e(\l^{j-1})) = -q\cdot \e(x_{i_1}^{e_1}\ldots x_{i_j}^{e_j})$, if $e_j = -1$. Here we use the fact that $\e(\l)=1$ since $\ell_{d+1}$ has zero linking number with $\ell_1, \ldots, \ell_d$.

Lemma \ref{redundant} allows us to discard some row of $A(q)$ other than the last without affecting the module presented.  Let $A^\e_\flat(q)$ be the resulting square matrix. The Alexander polynomial $\D_{L(q)}$ is the determinant of $A^\e_\flat(q)$, which we can write as a Laurent polynomial $P(t_1, \ldots, t_d, q)$. Note that $q$ appears linearly in $P$. 

The limit  $\lim_{q \to \infty} (1/q) P(t_1, \ldots, t_d, q)$ is the same as the determinant of the matrix obtained from $A^\e_\flat(q)$ by modifying the last row: setting $q$ equal to 1 and replacing  with $0$ the coefficient $1$ in the column corresponding to $x_n$. The modified row agrees with the Fox partial derivatives of the relation $\l$=1, the relation that arises from 0-framed surgery on $\ell_{d+1}$.

The determinant of the modified matrix is the Alexander polynomial 
$\D^\e_{\pi_1 M}$, where $M$ is obtained from the exterior $X_{L_0}$ via 
0-framed surgery on $\ell_{d+1}$, and $\e: \pi_1M\to \Z^d$ is induced by the augmentation of $\pi_L$ that is standard on $\pi_{L_0}$ but maps the class of $\ell_{d+1}$ trivially.

We have proved: 

\begin{theorem}\label{zerotwist} Assume that $L = \ell_1 \cup \cdots \cup \ell_{d+1}$ is an oriented $(d+1)$-component link in the 3-sphere $\S^3$ such that $\ell_{d+1}$ has zero linking number with each component of the sublink $L_0=\ell_1\cup \cdots \cup \ell_d$. Then there exists a polynomial $P(t_1, \ldots, t_d, q)$  with the following properties: 
\item{} (1) $P=P_0+qP_1$ with $P_0, P_1 \in\Z[t_1, \ldots, t_d]$; 
\item{} (2) For any positive integer $q$, $P(t_1, \ldots, t_d, q) \doteq \D_{L(q)}(t_1, \ldots, t_d)$; 
\item{} (3) $\lim_{q \to \infty} (1/q)P(t_1, \ldots, t_d, q)=P_1(t_1, \ldots, t_d)\doteq \D^\e_{\pi_1 M}(t_1, \ldots, t_d)$, 

\ni where $M$ is the 3-manifold obtained from the exterior of $L_0$ by 0-framed surgery on $\ell_{d+1}$, and $\e: \pi_1 M \to \Z^d$ is induced by the augmentation of $\pi_L$ that is standard on $\pi_{L_0}$ but maps a meridian of $\ell_{d+1}$ trivially.
\end{theorem}

\begin{remark} The polynomial $P(t_1, \ldots, t_d, q)$ provides a uniform normalization of the Alexander polynomials $\D_{L(q)}$. Using it, we can speak of the coefficients of $\D_{L(q)}$. \end{remark}


\begin{example} \label{borromean} Consider the Borromean rings, the oriented 3-component link in Figure 1 with Wirtinger generators indicated. The link $L(q)$ appears in Figure 2. 

The homotopy class $\l$ of the longitude of $\ell_{3}$ is $x_0x_4^{-1}$. Using Wirtinger relations, it can be written as $x_0 x_2 x_0^{-1} x_2^{-1}$. The Alexander matrix $A^\e$ for $L(q)$, with the first 5 rows corresponding to relations at numbered crossings, is 
$$\begin{pmatrix} -1 & 1-t_1& 0 & 0 & 0\\
                               0 & 0 & 0 & -1 & t_1 \\
                               0 & 1 & -1 & 1-t_2 & 0\\
                               0 & 0 & 0 & -1 &t_1\\
                               0 & 1 &-1 &0 &1-t_2\\
                               -q &0 &0 &0 & 1 \end{pmatrix}.$$ The Alexander polynomial $\D_{L(q)}(t_1, t_2)$ is equal to $q(t_1-1)(t_2-1)$. 
                               
The proof of Lemma \ref{redundant} shows that the second row can be deleted without affecting elementary ideals. Then as in the proof of Theorem \ref{zerotwist}, replacing the coefficient in the lowest right-hand entry with $0$ and setting $q=1$ produces a matrix 
$$\begin{pmatrix} -1 & 1-t_1& 0 & 0 & 0\\
                               0 & 1 & -1 & 1-t_2 & 0\\
                               0 & 0 & 0 & -1 &t_1\\
                               0 & 1 &-1 &0 &1-t_2\\
                               -1 &0 &0 &0 & 0 \end{pmatrix}.\quad \quad (*)$$
The determinant of the matrix (*) is equal to 
$(t_1-1)(t_2-1)$. The limit $\lim_{q\to \infty}\D_{L(q)}$ is $t_2-1$.

\end{example}

\begin{figure}
\begin{center}
\includegraphics[height=4.5 in]{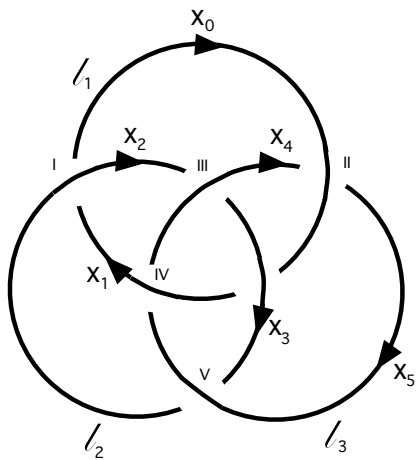}
\caption{Borromean rings}
\label{Borromean rings}
\end{center}
\end{figure}

\begin{figure}
\begin{center}
\includegraphics[height=1.5 in]{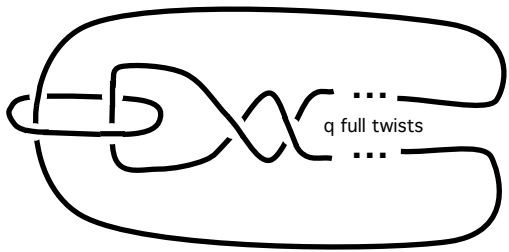}
\caption{Link $L(q)$}
\label{$L(q)$}
\end{center}
\end{figure}


The matrix (*) computed in Example \ref{borromean} can be computed for any link for which Theorem \ref{zerotwist} applies. The limit in statement (3) will vanish if and only if the matrix is singular. Equivalently,  the limit vanishes if and only if $P=P_0$, that is,
the polynomials $\D_{L(q)}$ do not depend on $q$. 

\begin{cor} \label{twist.cor} Assume the hypotheses of Theorem \ref{zerotwist}. The limit in statement (3) vanishes if and only if the sequence of polynomials $\D_{L(q)}$ is constant. \end{cor}

In the remainder of the section, we characterize homologically the case in which the limit of Theorem \ref{zerotwist} (3) vanishes.   

Let $X_L^\e$ denote the $\Z^d$-cover of the exterior of $L$ associated to the augmentation $\e$ in Theorem \ref{zerotwist}. The longitude of $\ell_{d+1}$ lifts to a closed oriented curve in $X_L^\e$; we fix a lift and regard it as an element of the $\Rd$-module $H_1 X_L^\e$. By abuse of notation, we let $\l$ denote this class.

The Alexander polynomial $\D^\e_L$ is equal to $\D_L(t_1,\ldots, t_d, 1)$ if $d=1$; it is equal to $\D_L(t_1,\ldots, t_d, 1) (t_1-1)$ if $d>1$ (see Proposition 7.3.10 of  \cite{kaw}, for example). If $\ell_{d+1}$ has zero linking number with $\ell_1, \ldots, \ell_d$, then the Torres conditions \cite{kaw} imply that $\D^\e_L =0$. Hence $\rk H_1 X_L^\e\ge 1$. We recall that the {\it rank} of an $\Rd$-module $H$, denoted by $\rk H$,  is the dimension of $H \otimes \Q(\Rd)$ regarded as a $\Q(\Rd)$-vector space, where $\Q(\Rd)$ is the field of fractions of $\Rd$. 

\begin{theorem} \label{characterize} Assume that $L = \ell_1 \cup \cdots \cup \ell_{d+1}$ is an oriented $(d+1)$-component link in the 3-sphere $\S^3$ such that $\ell_{d+1}$ has zero linking number with each component of $L_0=\ell_1\cup \cdots \cup \ell_d$. The following statements are equivalent.
\item{} (1) the sequence of polynomials $\D_{L(q)}$ is constant; 
\item{} (2) $\rk H_1 X_L^\e  > 1$ or
 $\l$ is a torsion element of $\in H_1 X_L^\e$. 
\end{theorem}

\begin{proof} 

Recall that $M$ is the 3-manifold obtained from $X_{L_0}$ by 0-framed surgery on $\ell_{d+1}$. Let $M^\e$ be the $\Z^d$-cover induced by $
\e$.  The homology $H_1 M^\e$ is the quotient of $H_1 X_L^\e$ by the submodule $\<\l\>$ generated by $\l$. 
By Theorem \ref{zerotwist} (3) and Corollary \ref{twist.cor}, the sequence of polynomials $\D_{L(q)}$ is constant if and only if 
$\rk H_1 M^\e > 0$. Hence statement (1)  is equivalent to the assertion that $H_1 M^\e \otimes \Q(\Rd) \cong  H_1X_L^\e/\<\l\>\otimes \Q(\Rd)$ is nontrivial. Since the latter module is isomorphic to $H_1 X_L^\e \otimes \Q(\Rd)/\<\l \otimes 1\>$, statements (1) and (2) are equivalent. 
\end{proof}


The restriction of $\e$ to $\pi_{L_0}$, which we also denote by $\e$, is the standard abelianization. When the Alexander polynomial of $L_0$ is nontrivial, the conclusion of Theorem \ref{characterize} simplifies.

\begin{cor} \label{torsionelement} Assume in addition to the hypotheses of Theorem \ref{characterize} that $\D_{L_0}\ne0$. Then
the sequence of polynomials $\D_{L(q)}$ is constant if and only if $\l$ is a torsion element of $H_1 X_L^\e$. \end{cor}

\begin{proof} The hypothesis $\D_{L_0}\ne 0$ implies that $H_2 X_{L_0}^\e =0$ \cite{kaw}. The long exact sequence of the pair $X_L^\e \subset X_{L_0}^\e$ together with Excision yields a short exact sequence: 

\begin{equation}\label{ses} 0 \to \Rd \to H_1X_L^\e \to H_1X_{L_0}^\e \to 0.\end{equation}

\ni Since $\rk H_1 X_{L_0}^\e=0$, tensoring with $\Q(\Rd)$ shows that $\rk H_1 X_L^\e =1$. Theorem \ref{characterize} completes the proof. 

\end{proof}

\begin{figure}
\begin{center}
\includegraphics[height=2.5 in]{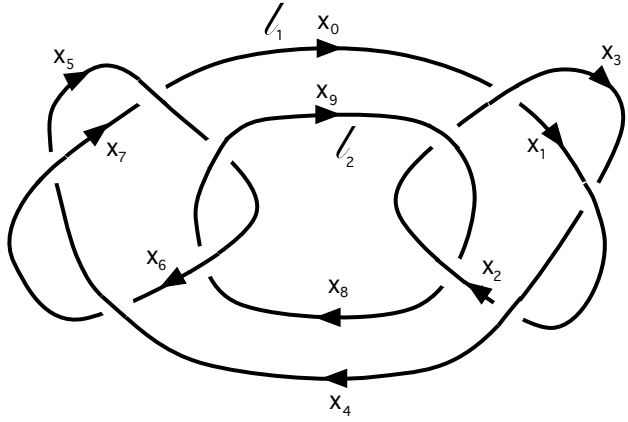}
\caption{Link $L$ with nontrivial torsion element $\l$}
\label{torsion}
\end{center}
\end{figure}
\begin{example} A torsion element of Corollary \ref{torsionelement} need not be trivial in $H_1 X_L^\e$, as we demonstrate. Consider the 2-component link $L$ with labeled Wirtinger generators in Figure \ref{torsion}. A straightforward calculation shows that $H_1 X_L^\e$ has module presentation 
$$\< x_3, x_5, x_9 \mid (t_1-1)x_9= (t_1^3-t_1^2+t_1)x_5, (t_1^2 - t_1+1)(x_5-x_3)\>.$$ The element $\l=x_6x_2^{-1}$ is conjugate in $\pi_L$ to $x_5 x_3^{-1}$ and it represents $x_5 - x_3$ in $H_1 X_L^\e$. Since $H_1 X_L^\e$ is isomorphic to the direct sum  $\< x_5, x_9 \mid (t_1 -1)x_9= (t_1^3-t_1^2+t_1)\l\> \oplus \{\l \mid (t_1^2-t_1+1)\l =0\>$, it is clear that $\l$ is a nontrivial torsion element in $H_1X_L^\e$. \end{example} 

The link in Example \ref{torsion} is a homology boundary link (see \cite{hillman}, p. 23). A link $L$ is a {\it homology boundary link} if there exist mutually disjoint properly embedded orientable surfaces $S_i$ 
in the link exterior $X_L$, corresponding to the components $\ell_i$, such that the boundary of $S_i$ is homologous to the longitude of the $i$th component. Since the linking number of any curve with the $i$th longitude is given by intersection number with $S_i$, each inclusion map $S_i \hookrightarrow X_L$ induces a trivial homomorphism on first homology.  

The following proposition implies that performing $1/q$-surgery on any component of a homology boundary $(d+1)$-component link produces a sequence of $d$-component links having the same Alexander polynomial. 

\begin {prop} \label{boundary.link} Let $L= \ell_1 \cup \cdots \cup \ell_{d+1}$ be an oriented $(d+1)$-component link as in Theorem \ref{characterize} . Assume that there exists a properly embedded orientable surface $S$ in $X_L$ with boundary homologous to the longitude of $\ell_{d+1}$ and such that the inclusion map $\iota: S \hookrightarrow X_L$ induces a trivial homomorphism on first homology. Then  the sequence of polynomials $\D_{L(q)}$ is constant. \end{prop}

\begin{proof} Since the image $\iota_*: \pi_1S\to \pi_1X_\ell$ is contained in the commutator subgroup  and since the cover $X_L^\e$ is abelian,  $\iota$ lifts to $X_L^\e$. The boundary of a lift represents $p \l$, for some $p = p(t_1, \ldots, t_d)$ such that $p(1, \ldots, 1) = 1$. Hence $\l$ is a torsion element of $H_1 X_L^\e$. Theorem \ref{zerotwist} and Corollary \ref{twist.cor} complete the proof. \end{proof}

We conclude the section with two examples and a conjecture. \bs

 \begin{example} Consider the oriented 2-component link $L = \ell_1 \cup \ell_2$ in Figure \ref{2twist}. A straightforward calculation shows that $\D_{L(q)} \doteq 2q t_1^2 + (1- 4 q) t_1 + 2 q$. Here $\lim_{q\to \infty} (1/q) \D_{L(q)} \doteq 2(t_1-1)^2$. 

By replacing the 2 full-twists in $\ell_1$ by $r$ full-twists, 
$\D_{L(q)}$ becomes $rq t_1^2 + (1- 2 r q) t_1 + 2 r q$ and 
$\lim_{q\to \infty} (1/q) \D_{L(q)} \doteq r (t_1-1)^2$. 

\end{example}

\begin{figure}
\begin{center}
\includegraphics[height=2 in]{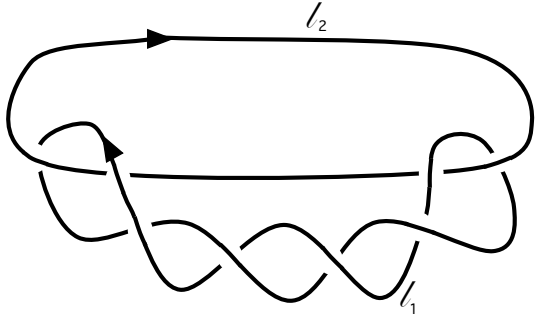}
\caption{$L= \ell_1 \cup \ell_2$ with linking number zero.}
\label{2twist}
\end{center}
\end{figure}

\begin{example} Consider the oriented 3-component link $L = \ell_1 \cup \ell_2 \cup \ell_3$ in Figure \ref{cyclotomic}. The sublink $L_0$ is a Hopf link, and $\ell_3$ has linking number zero with each of its components. A straightforward calculation shows that $\lim_{q\to \infty}(1/q)\D_{L(q)} \doteq  (t_1 t_2 -1)^2 (t_1^2 t_2^2 - t_1 t_2 + 1)$, a product of generalized cyclotomic polynomials.

\begin{figure}
\begin{center}
\includegraphics[height=3.5 in]{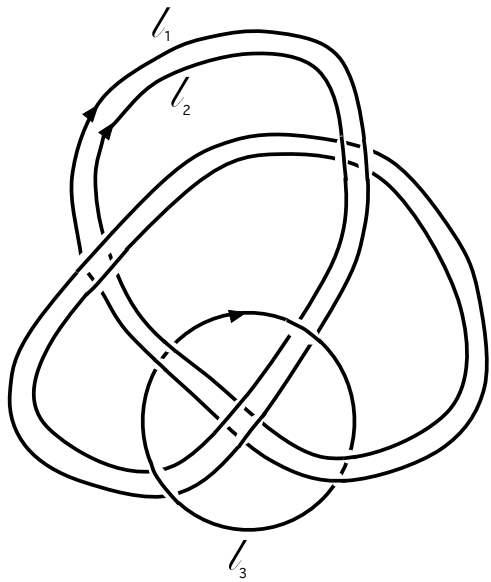}
\caption{Generalized cyclotomic factors arising.}
\label{cyclotomic}
\end{center}
\end{figure}

\end{example} 
 
\begin{conj} For any oriented link as in Theorem \ref{zerotwist}, 
$\lim_{q \to \infty} (1/q)\D_{L(q)}$ is an integer multiple of a product of generalized cyclotomic polynomials. 
\end{conj}  
 


\section{Twisted Alexander polynomials.} As in the previous section, let  $(G, \e)$ be a pair consisting of group $G$  with presentation \ref{presentation} and an epimorphism $\e$ to a nontrivial free abelian group $\Z^d$ generated by 
$t_1, \ldots, t_d$. We assume, as we did above, that $\e(x_0) = t_1$. 

Let $R$ be a Noetherian unique factorization domain, and $\rho: G \to GL_NR$  a linear representation. Below we will consider only the case that $R= \Z$ and the image of $\rho$ is a finite group of permutation matrices; in other words, $\rho$ is a representation of $G$ in the group $S_N$ of permutations of $\{1, \ldots, N\}$. 
Define
$$\e\otimes \rho: \Z[G] \to M_N(R[\Z^d])$$
by mapping $g \in G$ to $\e(g) \rho(g)$, and extending linearly. 

Twisted Alexander invariants generalize the (classical) Alexander polynomial by 
incorporating information from the representation $\rho$. They were
introduced by X.-S. Lin in \cite{lin} and later extended by many authors. Particularly relevant here are publications of Wada \cite{wada} and Kirk and Livingston \cite{kl}. The reader is referred to \cite{fvSURVEY} for a comprehensive survey of twisted Alexander invariants.

Wada's approach considers the {\it twisted Alexander matrix}

\begin{equation} \label{twisted.Alex.matrix} A^{\e \otimes \rho}=\bigg[(\e \otimes \rho) \bigg({{\partial r_i}\over {\partial x_j}}\bigg)\bigg]_{1\le i, j \le n}. \end{equation}
We regard $A^{\e \otimes \rho}$ as an $nN \times nN$ matrix over $R[\Z^d]$ by removing inner parentheses. The {\it Alexander-Lin} polynomial $D^{\e,  \rho}_G$ is the greatest common divisor of the maximal minors, well defined up to multiplication by units in $R[\Z^d]$. It is an invariant of the triple $(G, \e, x_0)$ (equivalence defined as for pairs $(G, \e)$ but respecting the distinguished group element $x_0$) and the conjugacy class of the representation $\rho$ (see  \cite{swALP}).  

Dividing $D^{\e,  \rho}_G$ by the determinant of $I - t_1\rho(x_0)$ eliminates the dependence on the distinguished element. The resulting rational function, well defined up to unit multiplication, is often called the {\it Wada invariant} of $(G, \e)$ and $\rho$, denoted here by $W^{\e, \rho}_G$. 

The Wada invariant has a homological interpretation with theoretical advantages over the combinatorial approach. 
Let $Y$ be a finite CW complex with $\pi_1Y \cong G$ having a single 0-cell as well as 1- and 2-cells corresponding to  the generators and relations in (\ref{presentation}). (When $G$ is the group of an oriented link, working with a Wirtinger presentation will ensure that $Y$ is homotopy equivalent to the link exterior.)  Let $\tilde Y$ denote the universal cover  of $Y$, with the structure of a CW complex that is lifted from $Y$. Consider the chain complex
$$C_*(Y; V[\Z^d]_\rho) = (R[\Z^d]\otimes_RV)\otimes_\rho C_*(\tilde Y),$$
where $V= R^N$ is a free module on which $G$ acts via $\rho$, while $C_*(\tilde Y)$ denotes the cellular chain complex of $\tilde Y$ with coefficients in $R$. The group $R[G]$ acts on the left by deck transformations. The tensor product $R[\Z^d]\otimes_R V$ has the structure of a right $R[G]$-module via
$$(p \otimes v)\cdot g = (\e(g)p)\otimes(v \rho(g)),\ {\rm for}\  g \in G.$$
The homology groups $H_*(Y; V[\Z^d])$ of the chain complex are finitely generated $R[\Z^d]$-modules. As above, elementary ideals
$E_k(H_*(Y; V[\Z^d]))$ are defined. The {\it module order} of $H_*(Y; V[\Z^d])$ -- that is, the greatest common divisor $\D_*$ of the elements of its 0th elementary ideal -- is an invariant of $(G, \e)$ and the conjugacy class of $\rho$. In \cite{kl} it is shown that $\D_1/\D_0$ is equal to the Wada invariant $W^{\e, \rho}_G$. 

When $G$ is the group of an oriented link $L$ in a homology 3-sphere, $\D_1$ is called the {\it twisted Alexander polynomial} of $L$ with respect to the augmentation $\e$ and representation $\rho$, which we denote by $\D_L^{\e, \rho}$. Similarly, we denote the Wada invariant by $W_L^{\e,\rho}$. When the augmentation is standard, we omit it from our notation. 

{\sl Henceforth we assume that $\rho$ is a finite-image permutation representation.}  It is easy to see that the denominator $\D_0$ of Wada's invariant is  a product of cyclotomic polynomials. In this case, the Mahler measures of $D_L^\rho$, $W_k^\rho$ and $\D_L^\rho$ are equal. Moreover,  Shapiro's lemma \cite{brown} implies that $H_*(X; V[\Z^d])$ is isomorphic to $H_1(\tilde X; \Z)$, where $\tilde X$ is the $N$-fold cover of the link exterior $X=X_L$ that is induced by $\rho$. 

An analogue of Theorem \ref{growth} was proven in \cite{swDTAP}.
For this we replace $M_\L$ by unbranched abelian cover $X_\L$ of $X_L$ corresponding to the finite-index subgroup $\L \subset \Z^d$. Since $\pi_1 X_\L$ is a subgroup of $G$, the representation $\rho$ restricts. Let $\tilde X_\L$ be the $N$-fold induced cover. We replace $b_\L$ by the order  of the torsion subgroup $H_1(\tilde X_\L;\Z)$, denoted by $b_{\L, \rho}$. Then Theorem 3.10 \cite{swDTAP}  implies that 
$$\log M(\D^{\rho}_L) = \limsup_{\< \L \> \to \infty} {1 \over |\Z^d/\L|} \log b_{\L, \rho}.$$

Formulating a theorem analogous to Theorem \ref{limit} is more problematic, since the groups $\pi_{L(q)}$ are quotients rather than subgroups of $\pi_L$. Our approach passes to an appropriate arithmetic subsequence of $L(q)$. 

Assume as in Section 3 that $L = \ell_1 \cup \cdots \cup \ell_{d+1}$ is an oriented link in $\S^3$ such that $\ell_{d+1}$ has nonzero linking number with some other component. Let $\rho_0: \pi_{L_0} \to GL_N\Z$ be a finite-image permutation representation, and let $r$ denote the order of $\rho(\l)$, where $\l \in \pi_L$ is the class of the longitude of $\ell_{d+1}$. For any positive integer $q$, the group $\pi_{L(rq)}$ is a quotient of $\pi_L$ by the relation $\l^{rq}x_n$, where $x_n$ is a merdianal generator of $\ell_{d+1}$. Consequently, the standard augmentation $\pi_{L_0} \to \Z^d$ extends to the standard augmentation of $\pi_{L(rq)}$, mapping the class of a meridian of $\ell_{d+1}$ to 
$(t_1^{-\l_1} \cdots t_d^{-\l_d})^{rq}$, where $\l_j$ is the linking number of $\ell_j$ and $\ell_{d+1}$. Moreover, 
$\rho_0$ induces a representation $\rho: \pi_{L(rq)} \to GL_N \Z$ that maps
a meridianal generator of $\ell_{d+1}$ trivially.

\begin{theorem} Let $L = \ell_1 \cup \cdots \cup \ell_{d+1}$ be an oriented link in $\S^3$ such that $\ell_{d+1}$ has nonzero linking number with some other component. Let $\l\in \pi_L$ be the class of the longitude of $\ell_{d+1}$.  Let $\rho_0: \pi_{L_0}\to GL_N\Z$ be a finite-image permutation representation and $\rho$ the extension to $\pi_L$ mapping $x_n$ trivially. Then $$\limsup_{q \to \infty} M(\D^\rho_{L(rq)}) = M(\D^ \rho_L),$$ where $r$ is the order of $\rho(\l)$.   \end{theorem} 

\begin{proof} Our proof generalizes arguments of \cite{kaw}, \cite{swJLM}. 
Let $\e: \pi_L \to \Z^d$ be the extension of the standard augmentation of $\pi_{L_0}$ that maps a meridian of $\ell_{d+1}$ to $(t_1^{-\l_1}\cdots t_d^{-\l_d })^{rq}$. 
 Consider the portion of the long exact sequence of the pair $X_L \subset X_{L(rq)}$: 

$$ \cdots \to\ H_2(X_{L(rq)}; V[\Z^d]) \ {\buildrel j_* \over \longrightarrow}\ H_2(X_{L(rq)}, X_L; V[\Z^d]) \ {\buildrel \partial_* \over \longrightarrow}\ H_1(X_{L}; V[\Z^d])$$ $$ \ {\buildrel i_* \over \longrightarrow}\ H_1(X_{L(rq)}; V[\Z^d])\ {\buildrel j_* \over \longrightarrow}\ H_1(X_{L(rq)}, X_L; V[\Z^d]))\ \to \cdots $$
By the excision isomorphism, $H_q(X_{L(rq)}, X_L; V[\Z^d]))$ is trivial unless 
$q=2$. 
One checks that 
$$H_2(X_{L(rq)}, X_L; V[\Z^d])) \cong  \bigr(\Z[\Z^d]/(t_1^{-\l_1}\cdots t_d^{-\l_d})^{rq} - 1)\bigr)^N $$ and hence its module order is a product $\Phi$ of generalized cyclotomic polynomials. 
Consequently, there exists a short exact sequence
$$ 0 \to\   \Z[\Z^d]^N/(f) \ {\buildrel \partial_* \over \longrightarrow}\ H_1(X_L; V[\Z^d]) \ {\buildrel i_* \over \longrightarrow}\ H_1(X_{L(rq)}; V[\Z^d]) \to 0, $$
where $f$ is a factor of $\Phi$. It follows that $\D^\rho_L$, the module order of $H_1(X_L; V[\Z^d])$, is the product of $f$ and $\D^\rho_{L(rq)}$, the module order  of
$H_1(X_{L(rq)}, V[\Z^d])$ (see \cite{kaw}, for example). Hence $\D^\rho_L$ vanishes if and only if $\D^\rho_{L(rq)}$ vanishes. In such a case, the conclusion of the theorem is trivial. 
Therefore we assume that $\D^\rho_{L(rq)}$ and $\D^\rho_L$ are nonzero.

The condition that $\D^\rho_{L(rq)}$ is nonzero implies that $H_2(X_{L(rq)}; V[\Z^d])=0$ (see Proposition 2 (5) of \cite{fvSURVEY}). The exact sequence above becomes short exact: 
$$ 0 \to\ H_2(X_{L(rq)}, X_L; V[\Z^d]) \ {\buildrel \partial_* \over \longrightarrow}\ H_1(X_L; V[\Z^d]) \ {\buildrel i_* \over \longrightarrow}\ H_1(X_{L(rq)}; V[\Z^d]) \to 0. $$
Hence $\D^\rho_{L(q)}$ is the product of $\Phi$ and the module order of $H_1(X_L; V[\Z^d])$. 

Recall that the Alexander-Lin polynomial $D_L^{\e, \rho}$ has the same Mahler measure as the twisted Alexander polynomial $\D_L^{\e, \rho}$. Working with the former is relatively easy: it is the determinant of the twisted Alexander matrix (\ref{twisted.Alex.matrix}) with $t_{d+1}$ replaced by 
$ (t_1^{-\l_1}\cdots t_d^{-\l_d })^{rq}$. 
We conclude that the Mahler measure of $\D^\rho_{L(rq)}$ is equal to that of $D^{\e, \rho}_L(t_1, \ldots, 
t_d, (t_1^{-\l_1}\cdots t_d^{-\l_d })^{rq})$. Corollary 3.2 of \cite{swJLM}, a consequence of a lemma of D. Boyd and W. Lawton \cite{lawton} (see also \cite{schmidt}), implies that  the limit of the Mahler measure of $D^{\e, \rho}_L(t_1, \ldots, 
t_d, (t_1^{-\l_1}\cdots t_d^{-\l_d })^{rq})$ as $q$ increases without bound is $M(D_L^{\rho})$, which is equal to the Mahler measure of $\D_L^\rho$. 
\end{proof} 

\begin{example} Consider the oriented 2-component link $L= \ell_1 \cup \ell_2$ in Figure \ref{pretzel}. The knot $L_0=\ell_1$ is the torus knot $5_1$, drawn as a pretzel knot. The knot $L(3)$ is the $(-2, 3, 7)$-pretzel knot. 
Let $\rho: \pi_{L_0} \to GL_5\Z$ be the permutation representation mapping 
$x_0$, $x_1$ and $x_2$ to matrices corresponding to the 5-cycles $(12345), (13542)$ and $(14235)$, respectively. The class $\lambda$ of the longitude of $\ell_2$ maps to a 3-cycle.  The twisted Alexander polynomial $\D_{L(3)}$ is the product of Lehmer's polynomial ${\cal L}(t)$ evaluated at $-t_1$ and a second, irreducible polynomial $f(t_1)$ of degree 34. The Mahler measure $M(f)$ is $1.7436\ldots.$ The fact that the (classical) Alexander polynomial is a factor of any twisted Alexander polynomial, whenever a finite-image permutation representation is employed, is a well-known consequence of the fact that the representation fixes a 1-dimensional subspace. (See discussion preceding Corollary \ref{same}.)

The limit $\lim_{q\to \infty} M(\D_{L(3q)})$ is approximately 4.18.

\end{example}

\begin{figure}
\begin{center}
\includegraphics[height=4 in]{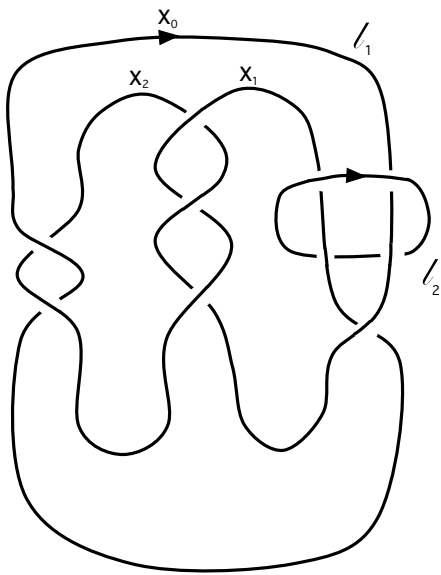}
\caption{Pretzel Link}
\label{pretzel}
\end{center}
\end{figure}

Theorem \ref{zerotwist} also generalizes.

\begin{theorem} \label{twisted.zerotwist} Assume that $L = \ell_1 \cup \cdots \cup \ell_{d+1}$ is an oriented $(d+1)$-component link in the 3-sphere $\S^3$ such that $\ell_{d+1}$ has zero linking number with each component of the sublink  $L_0= \ell_1 \cup \cdots \cup \ell_d$. Let $\rho_0: \pi_{L_0} \to GL_N\Z$ be a finite-image permutation representation and $\rho$ the extension to $\pi_L$ mapping $x_n$ trivially. Let $r$ denote the order of the permutation $\rho(\l)$.  Then
there exists a polynomial $P(t_1, \ldots, t_d, q)$  with the following properties: 
\item{} (1) $P=P_0+qP_1 + \cdots + q^N P_N$ with $P_0, \ldots , P_N \in\Z[t_1, \ldots, t_d]$; 
\item{} (2) For any positive integer $q$, $P(t_1, \ldots, t_d, q) \doteq \D^ \rho_{L(rq)}(t_1, \ldots, t_d)$; 
\item{} (3) $\lim_{q \to \infty} (1/q^N)P(t_1, \ldots, t_d, q) \doteq \D^{\e, \rho}_{\pi_L/\<\<\l^r\>\>}$, \smallskip

\ni where $\pi_L/\<\<\l^r\>\>$ is the quotient of $\pi_L$ by the normal subgroup generated by $\l^r$, $\e: \pi_1 M(r) \to \Z^d$ is induced by the augmentation of $\pi_L$ that is standard on $\pi_{L_0}$ but maps a meridian of $\ell_{d+1}$ trivially.
\end{theorem} 

\begin{proof} Construct the twisted Alexander matrix (\ref{twisted.Alex.matrix}) for $L$, and adjoin a block of $N$ rows corresponding to the relator $\l^{rq}x_n$. We obtain a twisted Alexander matrix $A^{\e, \rho}$ for $L(rq)$. Lemma \ref{redundant} enables us to remove a row of $N\times N$ blocks not involving meridians $\ell_{d+1}$, thereby obtaining a square matrix $A^{\e, \rho}_\flat$.

As in the proof of Theorem \ref{zerotwist}, we add nugatory crossings to $\ell_{d+1}$ to ensure that the word representing the longitude $\l$ does not contain the generator $x_n$. Consider the last row-block of $A^{\e, \rho}_\flat$. In the last column-block, corresponding to $x_n$, we see the $N \times N$ identity matrix. In each of the other nonzero column-blocks we find a matrix of the form $\pm q ( \sum_{j=0}^r \rho(\l)^j)Q$, where $Q$ is some permutation matrix. By conjugating the representation $\rho$, we can assume that the matrix $\rho(\l)$ has diagonal block form $B_1 \oplus \cdots \oplus B_s$, where each $B_i$ is a permutation matrix that acts transitively on its symbols.  We eliminate $Q$ by multiplying each block matrix in the column on the right by $Q^{-1}$,  a change of basis that permutes the generators of the column-block. (We continue to use the symbol $A^{\e, \rho}_\flat$ for this matrix.) 

Each 
$\sum_{j=0}^r \rho(\l)^j$  is a diagonal block matrix with blocks equal to constant matrices, the constant being $r/r_i$, where $r_i$ is the order of $B_i$. Dividing each row in the last row-block by $q$, and then taking the limit as $q \to \infty$, produces a new matrix $\bar A^{\e, \rho}_\flat$. It is identical with $A^{\e, \rho}_\flat$ everywhere except in the last row-block, where $q$ has been removed and the identity matrix in the last block has been replaced by the zero matrix. The new row-block corresponds to 
the relation $\l^r$, the relation that we would add to $\pi_L$ in order to obtain ${\pi_L/\<\<\l^r\>\>}$ instead of $\pi_{L(q)}$.

The twisted Alexander polynomials $\D_{L(q)}^{\rho}$ is  equal to 
$D_{L(q)}^{\rho} = \det A^{\e, \rho}_\flat$,  multiplied by the quotient $\D_0/(\det (I - t_1 \rho(x_0))$, where $\D_0$ is the module order of $H_0(X_{L(q)}; V[\Z^d])$. Since $\det (I - t_1 \rho(x_0))$ and $\D_0$ are independent of the relators of $\pi_{L(q)}$, the quotient does not involve $q$. Hence a polynomial $P(t_1, \ldots, t_d, q)$ satisfying (1) and (2) in the statement of the theorem. 

Similarly, the twisted polynomial $\D^{\e, \rho}_{{\pi_L/\<\<\l^r\>\>}}$ is the determinant of $\bar A^{\e, \rho}_\flat$ multiplied by the same quotient. 
The last claim of the theorem follows. 
 \end{proof}

Since the weight vector $(1, \ldots, 1)$ is an eigenvector with eigenvalue 1 for any any permutation representation, all finite-permutation representations have trivial, invariant 1-dimensional subrepresentations.  (For this reason, the classical Alexander polynomial is always a factor of the twisted Alexander polynomials considered here.) This observation, applied to the matrix $A^{\e, \rho}_\flat$ in the proof of Theorem \ref{twisted.zerotwist},  
establishes the following. 

\begin{cor} \label{same} Under the hypotheses of Theorem \ref{twisted.zerotwist}, the sequence of twisted Alexander polynomials $\D^\rho_{L(q)}$ is constant whenever the sequence of (untwisted) Alexander polynomials $\D_{L(q)}$ is constant. \end{cor}



\begin{thebibliography}{1}
 
 \bibitem{brown} K.S. Brown, Cohomology of Groups, Graduate Texts in Mathematics 87, Springer-Verlag, New York, 1994.

 
 \bibitem{dAGT} J.C. Dean, {\it Small Seifert-fibered Dehn surgery on hyperbolic knots}, {\sl Algebr. Geom. Topol.\ \bf 3} (2003),
435--472.
 
 \bibitem{ew} G. Everest and T. Ward, Heights of Polynomials and Entropy in Algebraic Dynamics, Springer-Verlag, London, 1999. 


\bibitem{fvSURVEY} S. Friedl and S. Vidussi, {\it A survey of twisted Alexander polynomials}, in {\sl The Mathematics of  Knots: Theory and  Applications}, edited by M. Banagl and D. Vogel, Springer-Verlag, Berlin-Heidelberg, 2011,
45--94. 

\bibitem{hillman} J. Hillman, Algebraic Invariants of Links, World Scientific, Singapore, 2002. 

\bibitem{kaw} A. Kawauchi, A Survey of Knot Theory, Birkh\"auser-Verlag, Tokyo, 1990. 


\bibitem{kl} P. Kirk and C. Livingston, {\it Twisted Alexander invariants, Reidemeister torsion and Casson-Gordon invariants}, {\sl Topology\ \bf 38} (1999), 635--661. 

\bibitem{lawton} W.M. Lawton, {\it A problem of Boyd concerning geometric means of polynomials}, {\sl J. Number Theory\ \bf16} (1983), 356--362.

\bibitem{le} T. Le, {\it Homology torsion growth and Mahler measure}, preprint, arXiv:1010.4199v2.

\bibitem{lin} X.S. Lin, {\it Representations of knot groups and twisted Alexander polynomials}, {\sl Acta Mathematica Sinica\ \bf17} (2001), 361--380 (based on 1990 Columbia University preprint). 


\bibitem{mpvMC} M.J. Mossinghoff, C.G. Pinner and J.D. Vaaler, {\it Perturbing polynomials with all their roots on the unit circle}, {\sl Mathematics of Computation\ \bf 67}, 1707--1726. 


\bibitem{schmidt} K. Schmidt, Dynamical Systems of Algebraic Origin, Birkh\"auser Verlag, Basel, 1995. 


\bibitem{swTOP2} D.S. Silver and S.G. Williams, {\it Mahler measure, links and homology growth}, {\sl Topology \bf 41} (2002), 979--991.


\bibitem{swJLM} D.S. Silver and S.G. Williams, {\it Mahler measure of Alexander polynomials}, {\sl J. London Math. Soc.\ \bf69} (2004), 767--782. 


\bibitem{swDTAP} D.S. Silver and S.G. Williams, {\it Dynamics of twisted Alexander polynomials}, {\sl Topology and its Applications \bf 156} (2009), 2795--2811.

\bibitem{swALP} D.S. Silver and S.G. Williams, {\it Alexander-Lin twisted polynomials}, {\sl J. Knot Theory Ramifications\ \bf 20} (2011), 427--434.


\bibitem{wada} M. Wada, {\it Twisted Alexander polynomial for finitely presented groups}, {\sl Topology\ \bf33} (1994), 241--256.



\end{thebibliography}
\end{document}